\documentclass[10pt,conference]{ieeeconf}

\usepackage{mathrsfs}
\usepackage{graphics} 
\usepackage{amsmath} 
\usepackage{amssymb}  

\usepackage{amsthm}
\usepackage{cite}
\usepackage{bm}
\usepackage{acronym}
\usepackage{paralist}
\usepackage{float}
\usepackage{color}
\usepackage{epstopdf}
\usepackage{multicol}
\usepackage{tikz}
\usepackage{hyperref}
\usepackage{graphicx}
\usepackage{soul}
\usepackage{mathtools}
\usepackage{siunitx}
\usepackage{empheq}
\usepackage{accents}
\usepackage[font=small, labelfont=small]{subcaption}
\usepackage[font=small, labelfont=small]{caption}
\usepackage{mdframed}
\usepackage[export]{adjustbox}
\renewcommand{\theenumi}{\roman{enumi}}

\makeatletter
\hypersetup{colorlinks=true}
\AtBeginDocument{\@ifpackageloaded{hyperref}
  {\def\@linkcolor{blue}
  \def\@anchorcolor{red}
  \def\@citecolor{red}
  \def\@filecolor{red}
  \def\@urlcolor{red}
  \def\@menucolor{red}
  \def\@pagecolor{red}
\begingroup
  \@makeother\`%
  \@makeother\=%
  \edef\x{%
    \edef\noexpand\x{%
      \endgroup
      \noexpand\toks@{%
        \catcode 96=\noexpand\the\catcode`\noexpand\`\relax
        \catcode 61=\noexpand\the\catcode`\noexpand\=\relax
      }%
    }%
    \noexpand\x
  }%
\x
\@makeother\`
\@makeother\=
}{}}
\makeatother

\newtheorem{Theorem}{Theorem}
\newtheorem{Lemma}{Lemma}
\newtheorem{Problem}{Problem}
\newtheorem{Remark}{Remark}

\newtheorem{Assumption}{Assumption}
\newtheorem{Definition}{Definition}

\DeclareMathOperator{\R}{\mathbb R}

\DeclareMathOperator*{\argmin}{arg\,min}

\newcommand{\bequ}{\begin{eqnarray}}
\newcommand{\eequ}{\end{eqnarray}}

\newcommand{\bb}{\boldsymbol}

\IEEEoverridecommandlockouts

\def\BibTeX{{\rm B\kern-.05em{\sc i\kern-.025em b}\kern-.08em
    T\kern-.1667em\lower.7ex\hbox{E}\kern-.125emX}}

\begin{document}

\title{\LARGE \bf Adaptation for Validation of a Consolidated Control Barrier Function based Control Synthesis}

\author{Mitchell Black$^1$ \and Dimitra Panagou$^2$
\thanks{The authors would like to acknowledge the support of the National Science Foundation award number 1931982.}
\thanks{$^1$Dept. of Aerospace Engineering,  Univ. of Michigan, 1320 Beal Ave, Ann Arbor, MI 48109, USA; \texttt{mblackjr@umich.edu}.}
\thanks{$^2$Dept. of Robotics and Dept. of Aerospace Engineering,  Univ. of Michigan, Ann Arbor, MI 48109, USA; \texttt{dpanagou@umich.edu}.}
}
\maketitle  


\begin{abstract}\label{abstract}
We develop a novel adaptation-based technique for safe control design in the presence of multiple control barrier function (CBF) constraints. Specifically, we introduce an approach for synthesizing any number of candidate CBFs into one consolidated CBF candidate, and propose a parameter adaptation law for the weights of its constituents such that the controllable dynamics of the consolidated CBF are non-vanishing. We then prove that the use of our adaptation law serves to certify the consolidated CBF candidate as valid for a class of nonlinear, control-affine, multi-agent systems, which permits its use in a quadratic program based control law. We highlight the success of our approach in simulation on a multi-robot goal-reaching problem in a crowded warehouse environment, and further demonstrate its efficacy experimentally in the laboratory via AION ground rovers operating amongst other vehicles behaving both aggressively and conservatively.
\end{abstract}


\section{Introduction}\label{sec.intro}
Since the arrival of control barrier functions (CBFs) to the field of safety-critical systems \cite{wielandallgower2007cbf}, much attention has been devoted to the development of their viability for safe control design \cite{ames2017control,Xiao2019HOCBF,cortez2019cbfmechsys}. As a set-theoretic approach founded on the notion of forward invariance, CBFs encode safety in that they ensure that any state beginning in a safe set remains so for all future time. In the context of control design, CBF conditions are often used as constraints in quadratic program (QP) based control laws, either as safety filters \cite{Chen2018Obstacle} or in conjunction with stability constraints (e.g. control Lyapunov functions) \cite{Garg2021Robust}. Their utility has been successfully demonstrated for a variety of safety-critical applications, including mobile robots \cite{Chen2021Guaranteed,Jankovic2021Collision}, unmanned aerial vehicles (UAVs) \cite{Xu2018Safe,Khan2020Cascaded}, and autonomous driving \cite{Black2022ffcbf,Yaghoubi2021RiskBoundedCBF}. But while it is now well-established that CBFs for controlled dynamical systems serve as certificates of safety, the verification of \textit{candidate} CBFs as \textit{valid} is in general a challenging problem.

Though for a single CBF there exist guarantees of validity under certain conditions for systems with either unbounded \cite{ames2017control} or bounded control authority \cite{Breeden2021InputConstraints, Xiao2022Adaptive}, these results do not generally extend to control systems seeking to satisfy multiple candidate CBF constraints. Recent approaches to control design in the presence of multiple CBF constraints have mainly circumvented this challenge by considering only one such constraint at a given time instance, either by assumption \cite{Cortez2022RobustMultiple} or construction in a non-smooth manner \cite{Glotfelter2017Nonsmooth,Huang2020SwitchedCBF}. In contrast, the authors of \cite{Lindemann2019CBFSTL} and \cite{Machida2021ConsensusCBF} each propose smoothly synthesizing one candidate CBF for the joint satisfaction of multiple constraints, but make no attempt to validate their candidate function. The problem of safe control design under a multitude of constraints is especially relevant in practical applications involving autonomous mobile robots, where the main challenge is in the robot completing its nominal objective while satisfying constraints related to collision avoidance with respect to obstacles both static and dynamic. 
\begin{figure}[!t]
    \centering
        \includegraphics[trim=2.5cm 18.75cm 3.5cm 0cm,clip,width=0.75\linewidth]{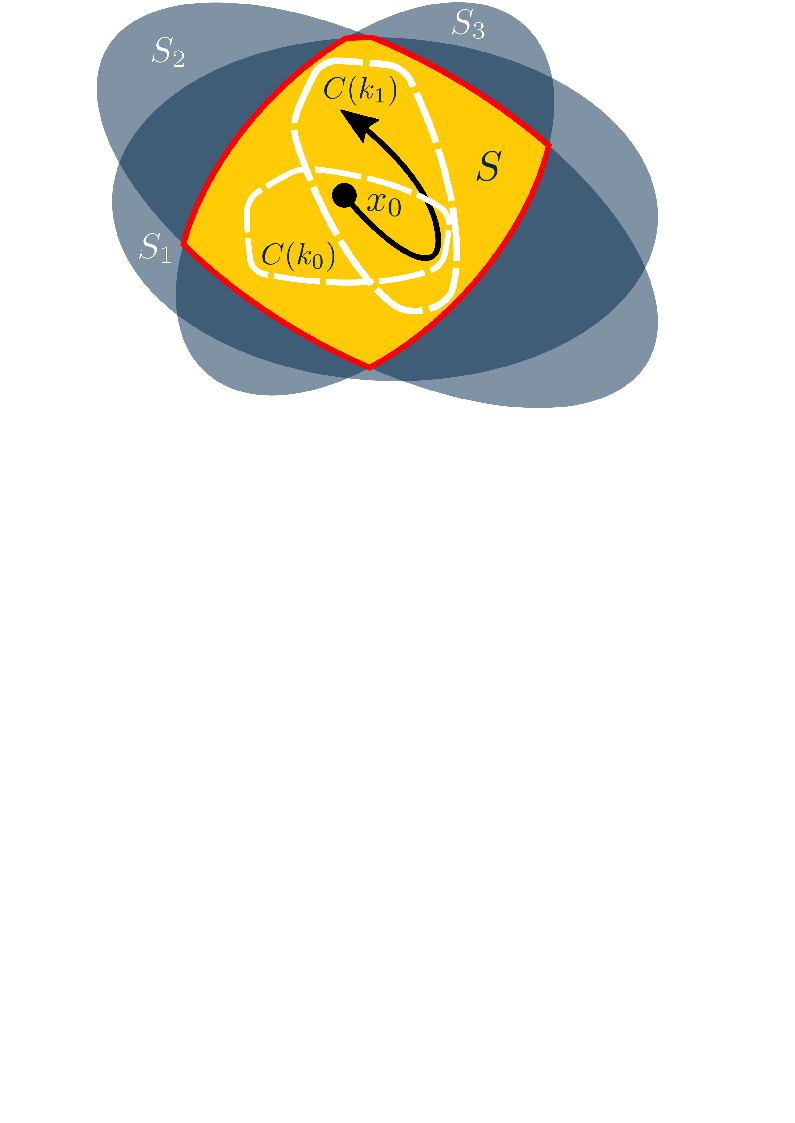}
    \caption{\small{Parameter adaptation for our C-CBF leads to a gain-dependent (and time-varying) controlled-invariant set $C(k) \subset S = \bigcap_{i=1}^c S_i$. $C(k)$ is shown here with a dotted white boundary for gains $k_0$ at time $t_0$ and $k_1$ at $t_1$.}}\label{fig.adaptive_safe_set}
    \vspace{-3mm}
\end{figure}


It is with this problem in mind that we propose a consolidated CBF (C-CBF) based approach to control design for multi-agent systems in the presence of both non-communicative and non-responsive (though non-adversarial) agents. Constructed by smoothly synthesizing any arbitrary number of candidate CBFs into one, our C-CBF defines a new super-level set that can under-approximate the intersection of its constituent sets arbitrarily closely (see Figure \ref{fig.adaptive_safe_set}). We further propose a parameter adaptation law for the weighting of the constituent functions, and prove that its use renders our C-CBF valid and the super-level set controlled invariant for the class of nonlinear, control-affine, multi-agent systems under consideration. And while various works have utilized parameter adaptation in the context of control for safety-critical systems, usually in an attempt to either learn \cite{Black2022Fixed,Lopez2021RaCBF} or compensate for \cite{Taylor2020Adaptive} unknown parameters in the system dynamics, our proposed adaptation law is the first to our knowledge to be used for the simultaneous verified satisfaction of multiple CBF constraints. To show the effectiveness of our proposed control formulation, we study a decentralized multi-robot goal-reaching problem in a crowded warehouse environment amongst non-responsive agents. As a practical demonstration, we tested our controller experimentally on a collection of ground rovers in the laboratory setting and found that it succeeded in safely driving the rovers to their goal locations amongst non-responsive agents behaving both aggressively and conservatively.

The paper is organized as follows. Section \ref{sec.prelims} introduces some preliminaries, including set invariance, optimization based control, and our first problem statement. In Section \ref{sec.consolidated_cbf}, we introduce the form of our C-CBF and propose a parameter adaptation law for rendering it valid. Sections \ref{sec.numerical_case_study} and \ref{sec.experimental_case_study} contain the results of our simulated and experimental case studies respectively, and in Section \ref{sec.conclusion} we conclude with final remarks and directions for future work.


\section{Mathematical Preliminaries}\label{sec.prelims}
We use the following notation throughout the paper. $\mathbb R$ denotes the set of real numbers. The set of integers between $i$ and $j$ (inclusive) is $[i..j]$. $\|\cdot\|$ represents the Euclidean norm.
A function $\alpha: \R \rightarrow \R$ is said to belong to class $\mathcal{K}_\infty$ if $\alpha(0)=0$ and $\alpha$ is increasing on the interval $(-\infty,\infty)$, A function $\phi: \R \times \R \rightarrow \R$ is said to belong to class $\mathcal{L}\mathcal{L}$ if for each fixed $r$ (resp. $s$), the function $\phi(r,s)$ is decreasing with respect to $s$ (resp. $r$) and is such that $\phi (r,s) \rightarrow 0$ for $s \rightarrow \infty$ (resp. $r \rightarrow \infty$). The Lie derivative of a function $V:\mathbb R^n\rightarrow \mathbb R$ along a vector field $f:\mathbb R^n\rightarrow\mathbb R^n$ at a point $x\in \mathbb R^n$ is denoted $L_fV(x) \triangleq \frac{\partial V}{\partial x} f(x)$. 

In this paper we consider a multi-agent system, each of whose $A$ constituent agents may be modelled by the following class of nonlinear, control-affine dynamical systems:
\begin{equation}\label{eq.agent_dynamics}
    \dot{\bb{x}}_i = f_i(\bb{x}_i(t)) + g_i(\bb{x}_i(t))\bb{u}_i(t), \quad \bb{x}_i(0) = \bb{x}_{i0}
\end{equation}
where $\bb{x}_i \in \R^n$ and $\bb{u}_i \in \mathcal{U}_i \subseteq \R^m$ are the state and control input vectors for the i$^{th}$ agent, with $\mathcal{U}_i$ the input constraint set, and where $f_i: \R^n \rightarrow \R^n$ and $g_i: \R^{n \times m} \rightarrow \R^n$ are known, locally Lipschitz, and not necessarily homogeneous $\forall i \in \mathcal{A} = [1..A]$. We denote the concatenated state vector as $\bb{x} = [\bb{x}_1,\hdots,\bb{x}_A]^T \in \R^N$, the concatenated control input vector as $\bb{u} = [\bb{u}_1,\hdots,\bb{u}_A]^T \in \mathcal{U} \subseteq \R^M$, and as such express the full system dynamics as
\begin{equation}\label{eq.multiagent_system}
    \dot{\bb{x}} = F(\bb{x}(t)) + G(\bb{x}(t))\bb{u}(t), \quad \bb{x}(0) = \bb{x}_0,
\end{equation}
where $F = [f_1,\hdots,f_A]^T: \R^N \rightarrow \R^N$ and $G = \textrm{diag}([g_1,\hdots,g_A]): \R^{M \times N} \rightarrow \R^N$. We assume that a (possibly empty) subset of the agents are communicative, denoted $j \in \mathcal{A}_c = [1..A_c]$, in the sense that they share information (e.g. states, control objectives, etc.) with one another, and that the remaining agents are non-communicative, denoted $k \in \mathcal{A}_n = [(A_c+1)..A]$, in that they do not share information, where $A_c \geq 0$ and $A_n = A - A_c \geq 0$ are the number of communicative and non-communicative agents respectively. We further assume that all agents are non-adversarial in that they do not seek to damage or otherwise deceive others, though there may be non-communicative agents which are non-responsive ($l \in \mathcal{A}_{n,n} \subseteq \mathcal{A}_n$) in that they do not actively avoid unsafe situations. 

\subsection{Safety and Forward Invariance}
Consider a set of safe states $S$ defined implicitly by a continuously differentiable function $h: \R^N \rightarrow \R$, as follows:
\begin{equation}\label{eq.safe_set}
    S = \{\bb{x} \in \R^N \; | \; h(\bb{x}) \geq 0\},
\end{equation}
where the boundary and interior of $S$ are denoted as $\partial S = \{\bb{x} \in \R^N \; | \; h(\bb{x}) = 0\}$ and $\textrm{int}(S) = \{\bb{x} \in \R^N \; | \; h(\bb{x}) > 0\}$ respectively. In many works (e.g. \cite{Gurriet2018Towards,Kolathaya2019ISSf}), the set $S$ defined by \eqref{eq.safe_set} is referred to as \textit{safe} if it is \textit{forward-invariant}, i.e. if $\bb{x}(0) \in S \implies \bb{x}(t) \in S$, $\forall t \geq 0$. Nagumo's Theorem provides a necessary and sufficient condition for rendering the set $S$ forward-invariant for the system \eqref{eq.multiagent_system}.
\begin{Lemma}[\hspace{-0.3pt}Nagumo's Theorem\cite{blanchini1999set}]\label{lemma.nagumos_theorem}
Suppose that there exists $\bb{u}(t) \in \mathcal{U}$ such that \eqref{eq.multiagent_system} admits a globally unique solution for each $\bb{x}_0 \in S$. Then, the set $S$ is forward-invariant for the controlled system \eqref{eq.multiagent_system} if and only if
\begin{equation}\label{eq.forward_invariance}
    L_Fh(\bb{x}) + L_Gh(\bb{x})\bb{u} \geq 0, \; \forall \bb{x} \in \partial S.
\end{equation}
\end{Lemma}
\noindent One way to render a set $S$ forward-invariant is to use CBFs in the control design.
\begin{Definition}\cite[Definition 5]{ames2017control}\label{def.cbf}
    Given a set $S \subset \R^N$ defined by \eqref{eq.safe_set} for a continuously differentiable function $h: \R^N \rightarrow \R$, the function $h$ is a \textbf{control barrier function} (CBF) defined on a set $D \supseteq S$ if there exists a Lipschitz continuous class $\mathcal{K}_\infty$ function $\alpha: \R \rightarrow \R$ such that, for all $\bb{x} \in D$,
    \begin{equation}\label{eq.cbf_condition}
        \sup_{\bb{u} \in \mathcal{U}}\left[L_Fh(\bb{x}) + L_Gh(\bb{x})\bb{u}\right] \geq -\alpha(h(\bb{x})).
    \end{equation}
\end{Definition}
In this paper, we assume that $\frac{\partial h}{\partial \bb{x}}$ is Lipschitz continuous so that $L_Fh(\bb{x})$ and $L_Gh(\bb{x})$ are likewise. In other works (e.g. \cite{Jankovic2022Multi}), the function $h$ responsible for defining $S$ is a CBF if there exists a class $\mathcal{K}_\infty$ function $\alpha$ satisfying
\begin{equation}\label{eq.alternate_cbf_condition}
    L_Gh(\bb{x}) = \mathbf{0}_{1\times M} \implies L_Fh(\bb{x}) + \alpha(h(\bb{x})) > 0.
\end{equation}
We note, however, that with unbounded control authority (i.e. $\mathcal{U} = \R^{M}$) a sufficient condition for the existence of some $\alpha \in \mathcal{K}_\infty$ satisfying \eqref{eq.cbf_condition}, and thus for $h$ to be a CBF, is $L_Gh(\bb{x}) \neq \mathbf{0}_{1\times M}$, $\forall \bb{x} \in S$, though this does not generally hold for a system with multiple CBF constraints.

\subsection{Control Design using CBFs}
Decentralized controllers, in which agents compute inputs based on local information, have found empirical success as a control strategy for multi-agent systems of the form \eqref{eq.multiagent_system} \cite{Borrmann2015Certificates,wang2017safety}.
The following is an example of one such controller for an agent $i\in\mathcal{A}$ with safety constraints encoded via $c>1$ candidate CBFs:
\begin{subequations}\label{eq.decentralized_cbf_qp_controller}
\begin{align}
    \bb{u}_i^* = \argmin_{\bb{u}_i \in \mathcal{U}_i} &\frac{1}{2}\|\bb{u}_i-\bb{u}_i^0\|^2 \label{subeq.decentralized_cbf_qp_objective}\\
    \textrm{s.t.} \quad &\forall s\in[1..c] \nonumber \\
    a_{s,i} + \bb{b}_{s,i}\bb{u}_{i} &\geq 0, \label{subeq.decentralized_cbf_qp_constraints}
\end{align}
\end{subequations}
where \eqref{subeq.decentralized_cbf_qp_objective} seeks to produce a control solution $\bb{u}_i^*$ that deviates minimally from some nominal input $\bb{u}_i^0$, and \eqref{subeq.decentralized_cbf_qp_constraints} encodes $c$ safety constraints of the form \eqref{eq.cbf_condition} via candidate CBFs $h_s$, where $a_{s,i} = L_{f_i}h_s + \alpha_s(h_s)$ and $\bb{b}_{s,i} = L_{g_i}h_s$.
Notably, for many classes of systems \eqref{eq.decentralized_cbf_qp_controller} is neither guaranteed to be feasible nor to preserve safety between agents \cite{Jankovic2021Collision}. When some subset of agents are able to communicate with one another, i.e. agents $j \in \mathcal{A}_c$ share information, their control inputs $\bb{u}_{\mathcal{A}_c} = [\bb{u}_1,\hdots,\bb{u}_{A_c}]^T$ may be computed in a centralized fashion as follows:
\begin{subequations}\label{eq.centralized_cbf_qp_controller}
\begin{align}
    \bb{u}_{\mathcal{A}_c}^* = \argmin_{\bb{u}_{\mathcal{A}_c} \in \mathcal{U}_{\mathcal{A}_c}} &\frac{1}{2}\|\bb{u}_{\mathcal{A}_c}-\bb{u}_{\mathcal{A}_c}^0\|^2 \label{subeq.centralized_cbf_qp_objective}\\
    \textrm{s.t.} \quad \forall j,k&\in\mathcal{A}_c, \; k \neq j \nonumber \\
    a_{s,j} + \bb{b}_{s,j}\bb{u}_{j} &\geq 0, \; \forall s \in [1..c_I], \label{subeq.centralized_cbf_qp_agentj} \\
    a_{s,jk} + \bb{b}_{s,j}\bb{u}_{j} + \bb{b}_{s,k}\bb{u}_{k} &\geq 0, \; \forall s \in [c_I + 1..c] \label{subeq.centralized_cbf_qp_interagent}
\end{align}
\end{subequations}
where $\bb{u}_{\mathcal{A}_c}^0 = [\bb{u}_1^0,\hdots,\bb{u}_{A_c}^0]^T$ is the nominal input vector shared amongst communicative agents, $\mathcal{U}_{\mathcal{A}_c} = \bigoplus_{j=1}^{A_c}\mathcal{U}_j$ is the Minkowski sum of their input constraint sets, \eqref{subeq.centralized_cbf_qp_agentj} denotes the $c_I \geq 0$ individual CBF constraints for agent $j$ (e.g. speed), and \eqref{subeq.centralized_cbf_qp_interagent} represents combinations of safety constraints between agents (e.g. collision avoidance), where $a_{s,jk} = L_{f_j}h_s + L_{f_k}h_s + \alpha_s(h_s)$, $\bb{b}_{s,j} = L_{g_j}h_s$, and $\bb{b}_{s,k} = L_{g_k}h_s$. When all agents are communicative, \eqref{eq.centralized_cbf_qp_controller} is guaranteed to be safe provided that it is feasible. 


A challenge when it comes to both \eqref{eq.decentralized_cbf_qp_controller} and \eqref{eq.centralized_cbf_qp_controller} is in satisfying all of the safety constraints simultaneously, especially when it comes to the design of $\alpha_s$.
In some recent works, authors have proposed setting $\alpha_s(h_s)=p_sh_s$ and including the parameters $p_s$ as decision variables in the QP \cite{Parwana2022Trust} (and thus an additional term $\sum_{s=1}^c\frac{1}{2}q_sp_s^2$ for $q_s>0$ in the objective function), but the performance of these approaches are still heavily dependent on the gains $q_s$. Other techniques have avoided the issue of multiple candidate CBFs by assuming that only one constraint is in need of satisfaction at once \cite{Cortez2022RobustMultiple} or by synthesizing a single non-smooth candidate CBF \cite{Glotfelter2017Nonsmooth, Huang2020SwitchedCBF}, both of which may lead to undesirable chattering behavior or the loss of existence and uniqueness of solutions. We seek to address this open problem, and require the following assumption to do so.
\begin{Assumption}\label{ass.nonempty_safe_set_intersection}
    The intersection of the safe sets $S_s$ for all $s \in [1..c]$ is non-empty, i.e. $S = \bigcap_{s=1}^cS_s \neq \emptyset$.
\end{Assumption}
\begin{Problem}\label{prob.consolidated_cbf}
    Given that Assumption \ref{ass.nonempty_safe_set_intersection} holds for a collection of $c>1$ candidate control barrier functions $h_s$ corresponding to safe sets $S_s$, design a consolidated control barrier function candidate $H: \R^N \times \R_+^c \rightarrow \R$ with constituent gains $\bb{k} = [k_1,\hdots,k_c]^T \in \R_+^c$ for the zero super-level set $C(\bb{k}) = \{\bb{x} \in \R^N \; | \; H(\bb{x}, \bb{k}) \geq 0\}$ such that $C(\bb{k}) \subseteq S$ for all $\bb{k}$ satisfying $0 < k_s < \infty$, $\forall s \in [1..c]$.
\end{Problem}

\section{Consolidated CBF based Control}\label{sec.consolidated_cbf}
In this section, we first introduce our proposed solution to Problem \ref{prob.consolidated_cbf}, a consolidated control barrier function (C-CBF) candidate that smoothly synthesizes multiple candidate CBFs into one, and then design a parameter adaptation law which renders the candidate C-CBF valid for safe control design.

\subsection{Consolidated CBFs}
Let the vector of $c>1$ candidate CBFs evaluated at a given state $\bb{x}$ be denoted $\bb{h}(\bb{x}) =  [h_1(\bb{x}) \; \hdots \; h_c(\bb{x})]^T \in \R^c$, and define a gain vector as $\bb{k} = [k_1 \; \hdots \; k_c]^T \in \R^c$, where $0< k_s < \infty$ for all $s \in [1..c]$. Our C-CBF candidate $H: \R^N \times \R_+^c \rightarrow \R$ is the following:
\begin{equation}\label{eq.consolidated_cbf}
    H(\bb{x}, \bb{k}) = 1 - \sum_{s=1}^c\phi\Big(h_s(\bb{x}), k_s\Big),
\end{equation}
where $\phi: \R_{\geq 0} \times \R_{\geq 0} \rightarrow \R_+$ belongs to class $\mathcal{L}\mathcal{L}$, is continuously differentiable, and satisfies $\phi(h_s,0)=\phi(0,k_s)=\phi(0,0)=1$. For example, the decaying exponential function, i.e. $\phi(h_s,k_s)=e^{-h_sk_s}$, satisfies these requirements over the domain $\R_{\geq 0} \times \R_{\geq 0}$. With $\phi$ possessing these properties, it follows then that the new zero super level-set $C(\bb{k}) = \{\bb{x} \in \R^N \; | \; H(\bb{x}, \bb{k}) \geq 0\}$ is a subset of $S$ (i.e. $C(\bb{k} \subset S$), where the level of closeness of $C(\bb{k})$ to $S$ depends on the choices of gains $\bb{k}$. This may be confirmed by observing that if any $h_s(\bb{x}) = 0$ then $H(\bb{x}) \leq 1 - 1 - \sum_{j=1, j\neq s}^c \phi(h_j(\bb{x}), k_j) < 0$, and thus for $H(\bb{x}) \geq 0$ it must hold that $h_s(\bb{x}) > 0$, for all $s \in [1..c]$.

As such, $H$ defined by \eqref{eq.consolidated_cbf} is a solution to Problem \ref{prob.consolidated_cbf}, i.e. $H$ is a C-CBF candidate. This implies via Lemma \ref{lemma.nagumos_theorem} that if $H$ is \textit{valid} over the set $C(\bb{k})$, then $C(\bb{k})$ is controlled invariant and thus the trajectories of \eqref{eq.multiagent_system} remain safe with respect to each constituent safe set $S_s$, $\forall s \in [1..c]$. By Definition \ref{def.cbf}, for a static gain vector (i.e. $\dot{\bb{k}} = \mathbf{0}_{c \times 1}$) the function $H$ is a CBF on the set $S$ if there exists $\alpha_H \in \mathcal{K}_\infty$ such that the following condition holds for all $\bb{x} \in S \supset C(\bb{k})$:
\begin{equation}\label{eq.ccbf_condition_static_k}
    L_FH(\bb{x}, \bb{k}) + L_GH(\bb{x}, \bb{k})\bb{u} \geq -\alpha_H(H(\bb{x}, \bb{k})),
\end{equation}
where from \eqref{eq.consolidated_cbf} it follows that
\begin{align}
    L_FH(\bb{x}) &= -\sum_{s=1}^c\frac{\partial \phi}{\partial h_s}L_Fh_c(\bb{x}), \label{eq.LfH_static_k} \\
    L_GH(\bb{x}) &= -\sum_{s=1}^c\frac{\partial \phi}{\partial h_s}L_Gh_c(\bb{x}) \label{eq.LgH_static_k}.
\end{align}
Again taking $\phi(h_s, k_s) = e^{-h_sk_s}$ as an example, we obtain that $\frac{\partial \phi}{\partial h_s} = -k_se^{-h_sk_s}$, in which case it is evident that the role of the gain vector $\bb{k}$ is to weight the constituent candidate CBFs $h_s$ and their derivative terms 
$L_Fh_s$ and $L_Gh_s$ in the CBF condition \eqref{eq.ccbf_condition_static_k}. Thus, a higher value $k_s$ indicates a weaker weight in the CBF dynamics, as the exponential decay overpowers the linear growth. Due to the combinatorial nature of these gains, for an arbitrary $\bb{k}$ there may exist some $\bb{x} \in C(\bb{k})$ such that $L_GH(\bb{x}) = \mathbf{0}_{1\times M}$, which may violate \eqref{eq.alternate_cbf_condition} and lead to the state exiting $C(\bb{k})$ (and potentially $S$ as a result). 
Using online adaptation of $\bb{k}$, however, it may be possible to achieve $L_GH(\bb{x}) \neq \mathbf{0}_{1\times M}$ for all $t \geq 0$, which motivates the following problem.
\begin{Problem}\label{prob.adaptation}
    Given a C-CBF candidate $H: \R^N \times \R_+^c \rightarrow \R$ defined by \eqref{eq.consolidated_cbf} and associated with the set $C(\bb{k})$, design an adaptation law $\dot{\bb{k}} = \kappa(\bb{x}, \bb{k})$ such that $L_GH \neq \mathbf{0}_{1\times M}$ for all $t \geq 0$.
\end{Problem}

\subsection{Adaptation for Control Synthesis}
Before proceeding with our main result, we require the following assumption.
\begin{Assumption}\label{ass.LgH_nonzero}
    The matrix of controlled candidate CBF dynamics $\bb{L}_g \in \R^{c \times M}$ is not all zero, i.e.
    \begin{equation}\label{eq.G_matrix}
        \bb{L}_g = \begin{bmatrix} L_gh_1 \\ \vdots \\ L_gh_c \end{bmatrix} \neq \mathbf{0}_{c \times M}. 
    \end{equation}
\end{Assumption}
We now present our main result, an adaptation law that solves Problem \ref{prob.adaptation} and thus renders $H$ a valid CBF for the set $C(\bb{k}(t))$, for all $t \geq 0$.
\begin{Theorem}\label{thm.kdot_HCBF}
    Suppose that there exist $c>1$ candidate CBFs $h_s: \R^N \rightarrow \R$ defining sets $S_s = \{\bb{x} \in \R^N \; | \; h_s(\bb{x}) \geq 0\}$, $\forall s \in [1..c]$, and that it is known that $\mathcal{U}=\R^M$. If $\bb{k}(0)$ is such that $L_GH \neq \mathbf{0}_{1 \times M}$ at $t=0$, then, under the ensuing adaptation law,
    \begin{subequations}\label{subeq.k_dot}
        \begin{align}
            \kappa(\bb{x}, \bb{k}) = \argmin_{\bb{\mu} \in \R^c} \; \frac{1}{2}(\bb{\mu} - \bb{\mu}_0)^T\bb{P}&(\bb{\mu} - \bb{\mu}_0) \\
            \mathrm{s.t.} \quad \quad \quad &\nonumber \\
            \bb{\mu} + \alpha_k(\bb{k}-\bb{k}_{min}) &\geq 0, \label{subeq.kdot.k_gr0} \\
            \bb{p}^T\bb{Q}\dot{\bb{p}} + \bb{p}^T\dot{\bb{Q}}\bb{p} + \alpha_{p}(h_{p}) &\geq 0, \label{subeq.kdot.vperp}
        \end{align}
    \end{subequations}
    the controlled CBF dynamics $L_GH \neq \mathbf{0}_{1 \times M}$ for all $t \geq 0$, and thus the function $H$ defined by \eqref{eq.consolidated_cbf} is a valid CBF for the set $C(\bb{k}(t)) = \{\bb{x} \in \R^N \; | \; H(\bb{x}, \bb{k}) \geq 0\}$, for all $t \geq 0$, where $\bb{P} \in \R^{c \times c}$ is a positive-definite gain matrix, $\alpha_k, \alpha_p \in \mathcal{K}_\infty$, $\bb{\mu}_0$ is the nominal $\dot{\bb{k}}$, $\bb{k}_{min} = [k_{1,min},\hdots,k_{c,min}]^T$ is the vector of minimum allowable values $k_{s,min} > 0$, and
    \begin{align}
        \bb{p} &= \left[\frac{\partial \phi}{\partial h_1} \; \hdots \; \frac{\partial \phi}{\partial h_c}\right]^T, \label{eq.p_vector} \\
        \bb{Q} &= \bb{I} - (\bb{N}\bb{N}^T)^T - \bb{N}\bb{N}^T - (\bb{N}\bb{N}^T)^T\bb{N}\bb{N}^T \label{eq.Q_matrix}
    \end{align}
    with $h_p = \frac{1}{2}\bb{p}^T\bb{Q}\bb{p} - \varepsilon$, $\varepsilon > 0$, and
    \begin{equation}\label{eq.N_basis_vectors}
        \bb{N} = [\bb{n}_1 \; \hdots \; \bb{n}_r],
    \end{equation}
    such that $\{\bb{n}_1,\hdots,\bb{n}_r\}$ constitutes a basis for the null space of $\bb{L}_g^T$, i.e. $\mathcal{N}(\bb{L}_g^T) = \mathrm{span}\{\bb{n}_1,\hdots,\bb{n}_r\}$, where $\bb{L}_g$ is given by \eqref{eq.G_matrix}.
\end{Theorem}
\begin{proof}
    
    
    First, given \eqref{eq.consolidated_cbf}, we have that
    \begin{align}
        \dot{H} &= -\sum_{s=1}^c\left(\frac{\partial \phi}{\partial h_s}\dot{h}_s + \frac{\partial \phi}{\partial k_s}\dot{k}_c\right) \nonumber \\
        &= \bb{p}^T \dot{\bb{h}} + \bb{q}^T \dot{\bb{k}} \nonumber \\
        &= \bb{p}^T (\bb{L}_f + \bb{L}_g\bb{u}) + \bb{q}^T \dot{\bb{k}} \nonumber
    \end{align}
    where $\bb{p}$ is given by \eqref{eq.p_vector}, 
    $\bb{L}_g$ by \eqref{eq.G_matrix}, $\bb{L}_f = [L_Fh_1 \; \hdots \; L_Fh_c]^T$, and $\bb{q} = [\frac{\partial \phi}{\partial k_1} \; \hdots \; \frac{\partial \phi}{\partial k_c}]^T$.
    As such, $L_FH = \bb{p}^T\bb{L}_f + \bb{q}^T\dot{\bb{k}}$ and $L_GH = \bb{p}^T\bb{L}_g$. With $\mathcal{U} = \R^M$, it follows that as long as $L_GH \neq \mathbf{0}_{1 \times M}$ it is possible to choose $\bb{u}$ such that $\dot{H}(\bb{x},\bb{u}) \geq -\alpha_H(H)$. We will now show that with $\dot{\bb{k}} = \kappa(\bb{x}, \bb{k})$ given by \eqref{subeq.k_dot} it holds that $L_GH \neq \mathbf{0}_{1 \times M}$ and thus $H$ is a CBF for $C(\bb{k}(t))$, for all $t \geq 0$.
    
    Since $L_GH = \bb{p}^T\bb{L}_g$, the problem of showing that $L_GH \neq \mathbf{0}_{1\times M}$ is equivalent to proving that $\bb{p} \notin \mathcal{N}(\bb{L}_g^T) = \mathrm{span}\{\bb{n}_1,\hdots,\bb{n}_r\}$. Since the vector $\bb{p}$ can be expressed as a sum of vectors perpendicular to and parallel to $\mathcal{N}(\bb{L}_g^T)$ (respectively $\bb{p}^\perp$ and $\bb{p}^\parallel$), it follows that $\bb{p} \notin \mathcal{N}(\bb{L}_g^T)$ as long as $\|\bb{p}^\perp\|>0$, where $\bb{p}^\perp = \left(\bb{I} - \bb{N}\bb{N}^T\right)\bb{p}$
    by vector projection, and $\bb{N}$ is given by \eqref{eq.N_basis_vectors}. Thus, a sufficient condition for $\bb{p} \notin \mathcal{N}(\bb{L}_g^T)$ is that
    \begin{equation}\label{eq.v_condition}
        \frac{1}{2}\|(\bb{I}-\bb{N}\bb{N}^T)\bb{p}\|^2 = \frac{1}{2}\bb{p}^T\bb{Q}\bb{p} > \varepsilon
    \end{equation}
    for some $\varepsilon>0$, where $\bb{Q}$ is given by \eqref{eq.Q_matrix}. Then, by defining a function $h_p = \frac{1}{2}\bb{p}^T\bb{Q}\bb{p} - \varepsilon$, it follows from \eqref{eq.cbf_condition} that when \eqref{eq.v_condition} is true at $t=0$, it is true $\forall t \geq 0$ as long as \eqref{subeq.kdot.vperp} holds. 
    
    Therefore, gains $\bb{k}$ adapted according to the law \eqref{subeq.k_dot} are guaranteed to result in $L_GH \neq \mathbf{0}_{1 \times M}$. Thus, $H$ is a CBF for the set $C(\bb{k}(t))$, for all $t \geq 0$. This completes the proof.
\end{proof}
\begin{Remark}
    With $\bb{Q}$ depending on basis vectors spanning $\mathcal{N}(\bb{L}_g^T)$, it is not immediately obvious under what conditions $\dot{\bb{Q}}$ is continuous (or even well-defined). Prior results show that if the rank of $\mathcal{N}(\bb{L}_g^T)$ is constant then $\dot{\bb{Q}}$ varies continuously $\forall \bb{x} \in B_{\epsilon}(\bb{x})$ \cite{Coleman1984Orthonormal}, but analytical derivations of $\dot{\bb{Q}}$ are not available to the best of our knowledge. In practice, we observe that the rank of $\mathcal{N}(\bb{L}_g^T)$ is indeed constant, and we approximate $\dot{\bb{Q}}$ numerically using finite-difference methods. 
\end{Remark}

With $H$ consolidating the many constituent constraints into one CBF condition, we can then replace the centralized CBF-QP controller \eqref{eq.centralized_cbf_qp_controller} with the following:
\begin{subequations}\label{eq.proposed_centralized}
\begin{align}
    \bb{u}_{\mathcal{A}_c}^* &= \argmin_{\bb{u}_{\mathcal{A}_c} \in \mathcal{U}_{\mathcal{A}_c}} \frac{1}{2}\|\bb{u}_{\mathcal{A}_c}-\bb{u}_{\mathcal{A}_c}^0\|^2 \label{eq.proposed_centralized_cbf_qp_objective}\\
    &\quad \quad \textrm{s.t.} \nonumber \\
    a &+\bb{b}\bb{u}_{\mathcal{A}_c} \geq 0, \label{eq.proposed_centralized_cbf_qp_constraints}
\end{align}
\end{subequations}
where $a = L_FH + \alpha_H(H)$ and $\bb{b} = L_GH_{[i \in \mathcal{A}_c]}$. If all agents are communicative, i.e. $\mathcal{A}_c = \mathcal{A}$, then since $H$ is a CBF for the set $C(\bb{k}(t)) \subset S$, for all $t \geq 0$, the system trajectories are guaranteed to stay within $C(\bb{k}(t)) \subset S$ and thus remain safe. In the presence of non-communicative agents, we replace the decentralized CBF-QP controller \eqref{eq.decentralized_cbf_qp_controller} with
\begin{subequations}\label{eq.proposed_decentralized}
\begin{align}
    \bb{u}_i^* &= \argmin_{\bb{u}_i \in \mathcal{U}_{i}} \frac{1}{2}\|\bb{u}_{i}-\bb{u}_{i}^0\|^2 \label{eq.proposed_decentralized_qp_objective}\\
    &\textrm{s.t.} \nonumber \\
    a &+ \bb{b}_i\bb{u}_i \geq d, \label{eq.proposed_decentralized_qp_constraints}
\end{align}
\end{subequations}
where $\bb{b}_i = L_GH_{[mi:m(i+1)]}$, i.e. the portion of the dynamics of $H$ that agent $i$ controls, and $d = e^{-rH}\max_{\bb{u} \in \mathcal{U}}\sum_{j=1,j \neq i}^AL_GH_{[jm:j(m+1)]}u_j$, where $r>0$. While for the case of unbounded control authority $d$ is similarly unbounded, in practice it is reasonable to assume that agents have limited control authority and thus to use \eqref{eq.proposed_decentralized} assuming some bounded $\mathcal{U}$.


\section{Multi-Robot Numerical Study}\label{sec.numerical_case_study}
In this section, we demonstrate our C-CBF controller on a decentralized multi-robot goal-reaching problem. 

Consider a collection of 3 non-communicative, but responsive robots ($i \in \mathcal{A}_n \setminus \mathcal{A}_{n,n}$) in a warehouse environment seeking to traverse a narrow corridor intersected by a passageway occupied with 6 non-responsive agents ($j \in \mathcal{A}_{n,n}$). The non-responsive agents may be e.g. humans walking or some other dynamic obstacles. Let $\mathcal{F}$ be an inertial frame with a point $s_0$ denoting its origin, and assume that each robot may be modeled according to a dynamic extension of the kinematic bicycle model described by \cite[Ch. 2]{Rajamani2012VDC}, provided here for completeness:
\begin{subequations}\label{eq.dynamic_bicycle_model}
\begin{align}
    \dot{x}_i &= v_{i}\left(\cos{\psi_i} - \sin{\psi_i}\tan{\beta_i}\right) \label{eq: dyn x} \\
    \dot{y}_i &= v_{i}\left(\sin{\psi_i} + \cos{\psi_i}\tan{\beta_i}\right) \label{eq: dyn y} \\
    \dot{\psi}_i &= \frac{v_{i}}{l_r}\tan{\beta_i} \label{eq: dyn psi} \\
    \dot{\beta}_i &= \omega_i \\
    \dot{v}_{i} &= a_{i},
\end{align}
\end{subequations}
where $x_i$ and $y_i$ denote the position (in m) of the center of gravity (c.g.) of the i$^{th}$ robot with respect to $s_0$, $\psi_i$ is the orientation (in rad) of its body-fixed frame, $\mathcal{B}_i$, with respect to $\mathcal{F}$, $\beta_i$ is the slip angle\footnote{$\beta_i$ is related to the steering angle $\delta_i$ via $\tan{\beta_i} = \frac{l_r}{l_r+l_f}\tan{\delta_i}$, where $l_f+l_r$ is the wheelbase with $l_f$ (resp. $l_r$) the distance from the c.g. to the center of the front (resp. rear) wheel.} (in rad) of the c.g. of the vehicle relative to $\mathcal{B}_i$ (assume $|\beta_i|<\frac{\pi}{2}$), and $v_{i}$ is the velocity of the rear wheel with respect to $\mathcal{F}$. The state of robot $i$ is denoted $\bb{z}_i = [x_i \; y_i \; \psi_i \; \beta_i \; v_{i}]^T$, and its control input is $\bb{u}_i=[a_{i} \; \omega_i]^T$, where $a_{i}$ is the acceleration of the rear wheel (in m/s$^2$), and $\omega_i$ is the angular velocity (in rad/s) of $\beta_i$.

The challenges of this scenario relate to preserving safety despite multiple non-communicative and non-responsive agents present in a constrained environment. A robot is safe if it 1) obeys the speed restriction, 2) remains inside the corridor area, and 3) avoids collisions with all other robots. Speed is addressed with the following candidate CBF:
\begin{align}
    h_{v}(\bb{z}_i) &= s_M - v_i, \label{eq.speed_cbf}
\end{align}
where $s_M > 0$, while for corridor safety and collision avoidance we used forms of the relaxed future-focused CBF introduced in \cite{Black2022ffcbf} for roadway intersections, namely
\begin{align}
\begin{split}\label{eq.corridor_cbf}
    h_{c}(\bb{z}_i) &= (m_L(x_i + \dot{x}_i) + b_L - (y_i + \dot{y}_i)) \cdot \\ &\quad \; (m_R(x_i + \dot{x}_i) + b_R - (y_i + \dot{y}_i)) 
\end{split}\\
\begin{split}\label{eq.rff_cbf}
    h_{r}(\bb{z}_i,\bb{z}_j) &= D(\bb{z}_i,\bb{z}_j,t + \hat\tau)^2 \\ & \quad + \epsilon D(\bb{z}_i,\bb{z}_j,t)^2 - (1 + \epsilon)(2R)^2, 
\end{split}   
\end{align}
where \eqref{eq.corridor_cbf} prevents collisions with the corridor walls (defined as lines in the $xy$-plane via $m_L,b_L,m_R,b_R \in \R$), and \eqref{eq.rff_cbf} prevents inter-robot collisions and is defined $\forall i \in \mathcal{A}_n \setminus \mathcal{A}_{n,n}$, $\forall j \in \mathcal{A}_n$, where $\epsilon>0$, $D(\bb{z}_i, \bb{z}_j, t_a)$ is the Euclidean distance between agents $i$ and $j$ at arbitrary time $t_a$, and $\hat\tau$ denotes the time in the interval $[0,T]$ at which the minimum inter-agent distance will occur under constant velocity future trajectories. For a more detailed discussion on future-focused CBFs, see \cite{Black2022ffcbf}.
As such, \eqref{eq.speed_cbf}, \eqref{eq.corridor_cbf}, and \eqref{eq.rff_cbf} define the sets 
\begin{align}
    S_{v,i} &= \{\bb{z}_i \in \R^n \; | \; h_{v}(\bb{z}_i) \geq 0\}, \nonumber \\
    S_{c,i} &= \{\bb{z}_i \in \R^n \; | \; h_{c}(\bb{z}_i) \geq 0\}, \nonumber \\
    S_{r,i} &= \bigcap\limits_{j=1,j\neq i}^A \{\bb{z} \in \R^N \; | \; h_{r}(\bb{z}_i, \bb{z}_j) \geq 0\}, \nonumber
\end{align}
the intersection of which constitutes the safe set for agents $i$, i.e. $S_i(t) = S_{v,i} \cap S_{c,i} \cap S_{r,i}$. 

We control robots $i \in \mathcal{A}_n \setminus \mathcal{A}_{n,n}$ using a C-CBF based decentralized controller of the form \eqref{eq.proposed_decentralized} with constituent functions $h_c$, $h_s$, $h_r$, an LQR based nominal control input (see \cite[Appendix 1]{Black2022ffcbf}), and initial gains $\bb{k}(0) = \mathbf{1}_{10 \times 1}$. The non-responsive agents used a similar LQR  controller to move through the passageway in pairs of two, with the first two pairs passing through the intersection without stopping and the last pair stopping at the intersection before proceeding.

As shown in Figure \ref{fig.warehouse_robot_trajectories}, the non-communicative robots traverse both the narrow corridor and the busy intersection to reach their goal locations safely. The trajectories of the gains $\bb{k}$ for each warehouse robot are shown in Figure \ref{fig.kgains_warehouse}, while their control inputs are depicted in Figure \ref{fig.warehouse_control_inputs}. The CBF time histories for the constituent and consolidated functions are highlighted in Figures \ref{fig.warehouse_cbfs} and \ref{fig.warehouse_ccbfs} respectively, and show that the C-CBF controllers maintained safety at all times. 



\begin{figure}[!ht]
    \centering
        \includegraphics[clip,width=0.9\linewidth]{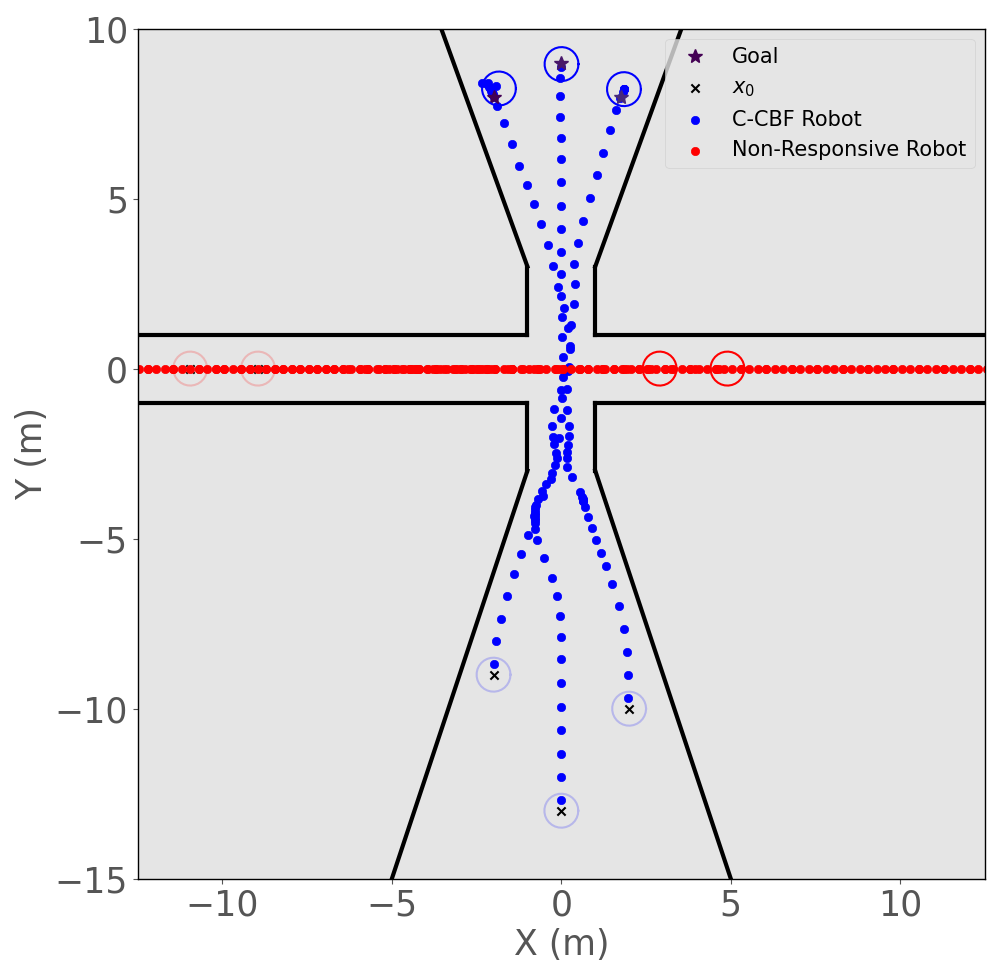}
    \caption{\small{XY paths for the warehouse robots (blue) and non-responsive agents (red) in the warehouse control problem.}}\label{fig.warehouse_robot_trajectories}
    \vspace{-3mm}
\end{figure}
\begin{figure}[!ht]
    \centering
        \includegraphics[clip,width=0.95\columnwidth]{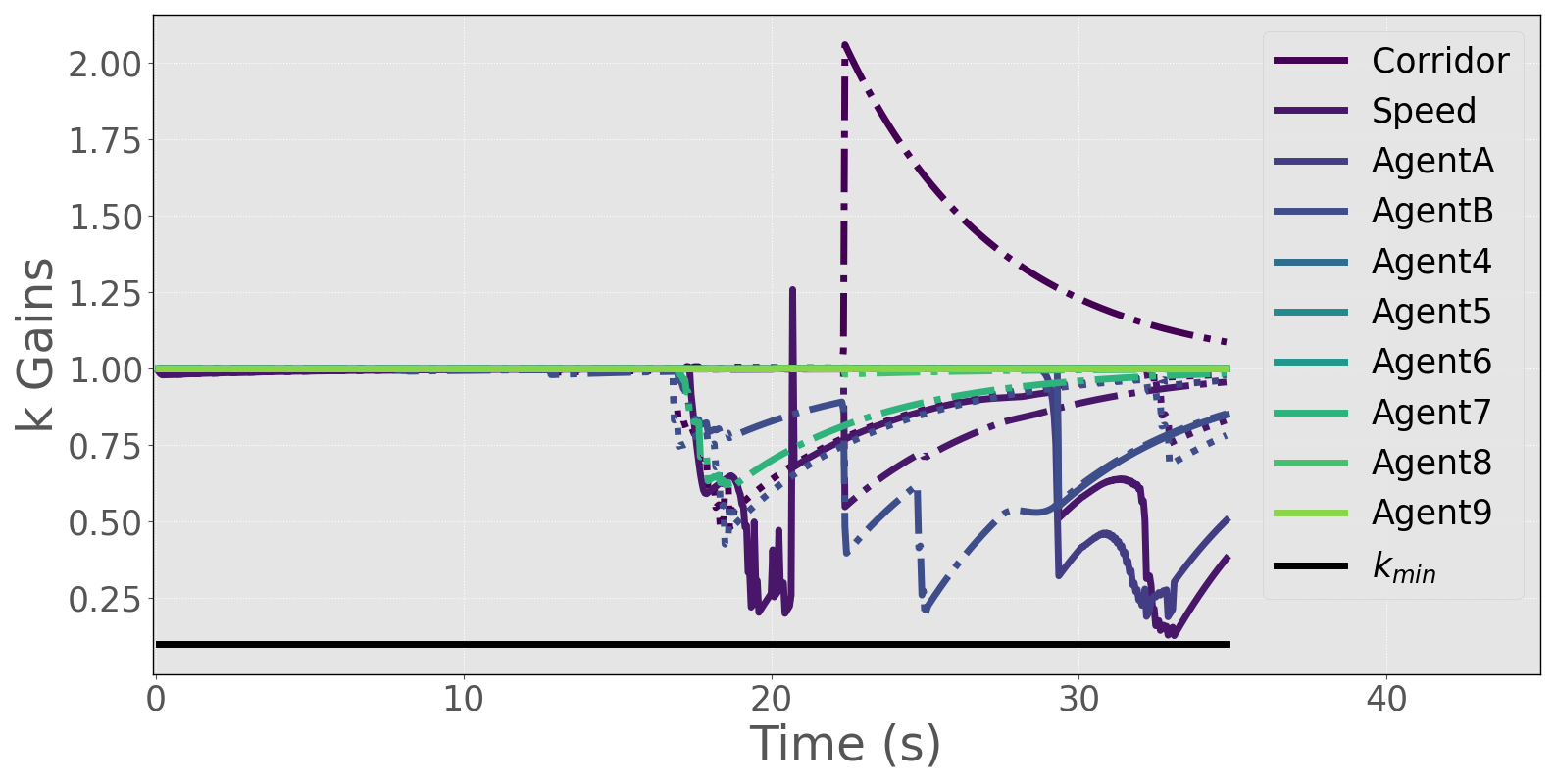}
    \caption{\small{Gains $\bb{k}$ for the C-CBF controllers in the warehouse study. Robot 1 denoted with solid lines, dotted for robot 2, dash-dots for robot 3. AgentA and AgentB denote the other two non-communicative robots from the perspective of one (e.g. AgentA=Agent1 and AgentB=Agent3 for robot 2). }}\label{fig.kgains_warehouse}
    \vspace{-3mm}
\end{figure}
\begin{figure}[!ht]
    \centering
        \includegraphics[clip,width=0.95\columnwidth]{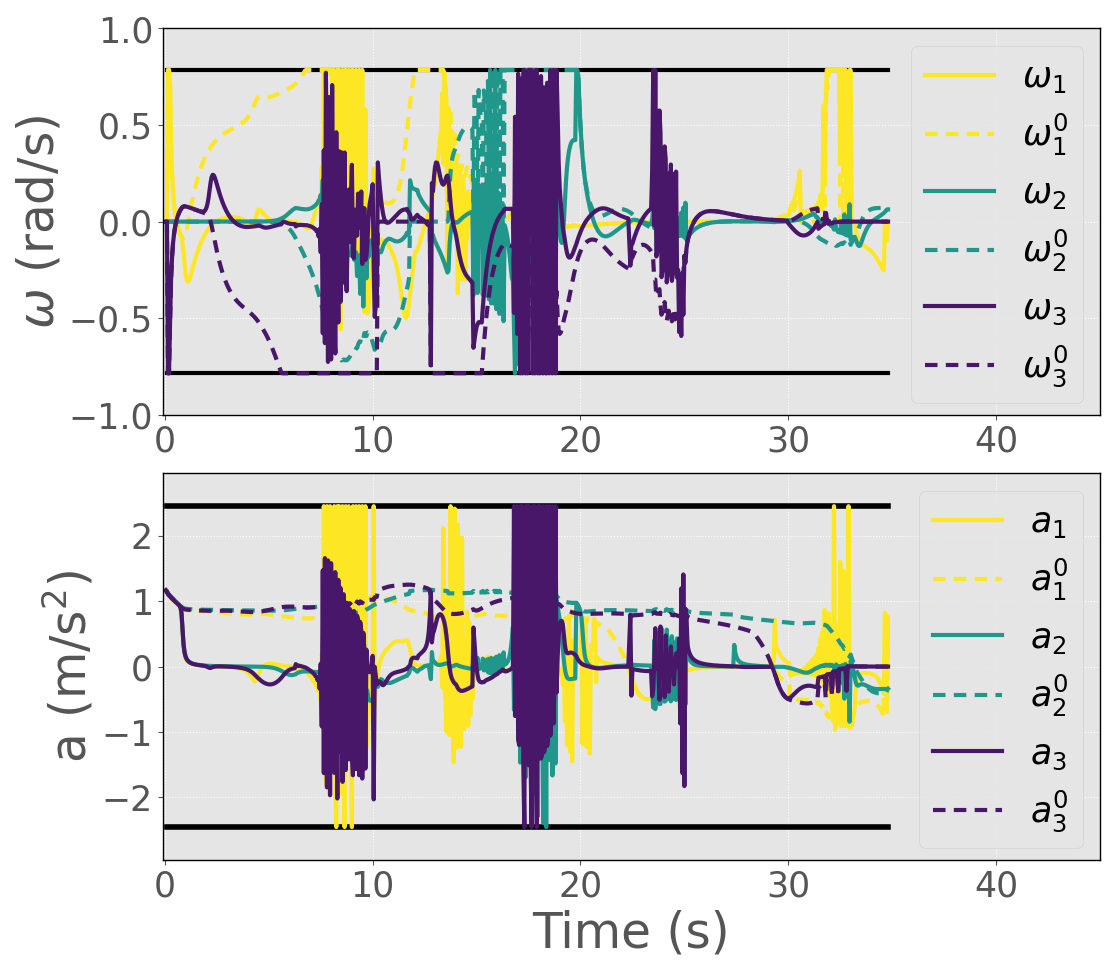}
    \caption{\small{Warehouse robot controls: accel. ($a$) and slip angle rate ($\omega$).}}\label{fig.warehouse_control_inputs}
    \vspace{-3mm}
\end{figure}
\begin{figure}[!ht]
    \centering
        \includegraphics[clip,width=0.95\columnwidth]{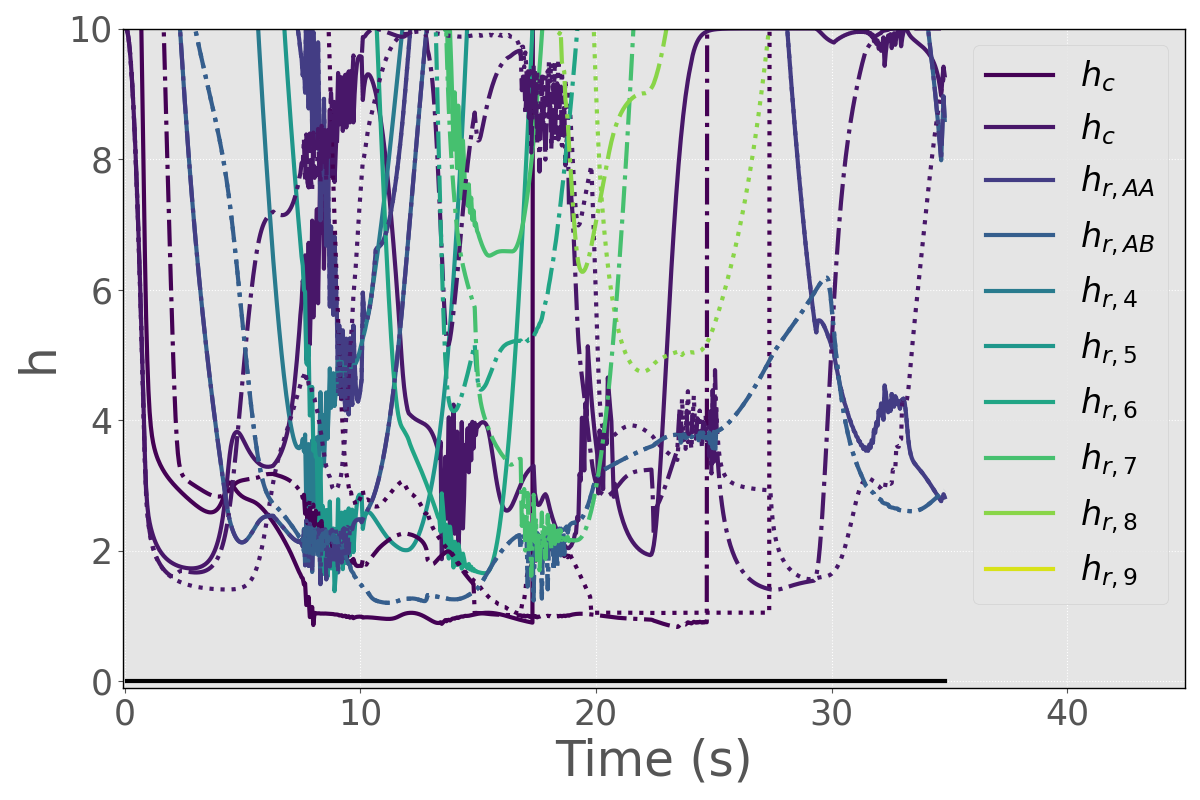}
    \caption{\small{Evolution of warehouse robot constituent CBF candidates, $h_s$ $\forall s \in [1..c]$, synthesized to construct C-CBF.}}\label{fig.warehouse_cbfs}
    \vspace{-3mm}
\end{figure}
\begin{figure}[!ht]
    \centering
        \includegraphics[clip,width=0.95\columnwidth]{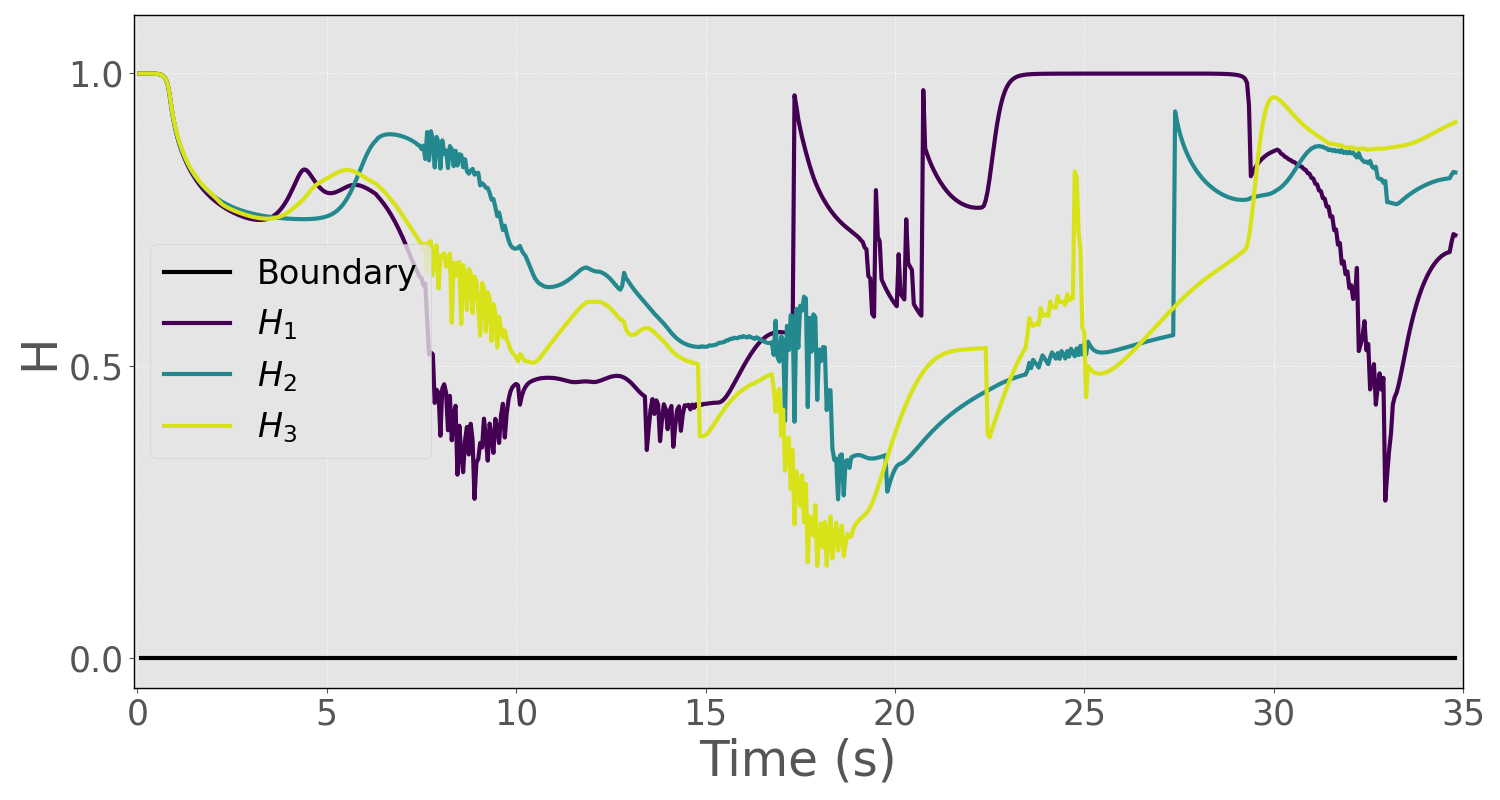}
    \caption{\small{Evolution of C-CBF $H$ for warehouse robots 1, 2, and 3.}}\label{fig.warehouse_ccbfs}
    \vspace{-1mm}
\end{figure}


\section{Experimental Case Study}\label{sec.experimental_case_study}
For experimental validation of our approach, we used an AION R1 UGV ground rover as an ego vehicle in the laboratory setting and required it to reach a goal location in the presence of two non-responsive rovers: one static and one dynamic. We modeled the rovers as bicycles using \eqref{eq.dynamic_bicycle_model}, and sent angular rate $\omega_i$ and velocity $v_i$ (numerically integrated based on the controller's acceleration output) commands to the rovers' on-board PID controllers. The ego rover used our proposed C-CBF \eqref{eq.proposed_decentralized} with constituent candidate CBFs \eqref{eq.speed_cbf} (with $s_M = 1$ m/s) and the rff-CBF defined in \eqref{eq.rff_cbf} for collision avoidance. The nominal input to the C-CBF controller was the LQR law from the warehouse robot example, as was the controller used by the dynamic non-responsive rover. A Vicon motion capture system was used for position feedback, and the state estimation was performed by extended Kalman filter via the on-board PX4.
\begin{figure}[!t]
    \centering
        \includegraphics[clip,width=\columnwidth]{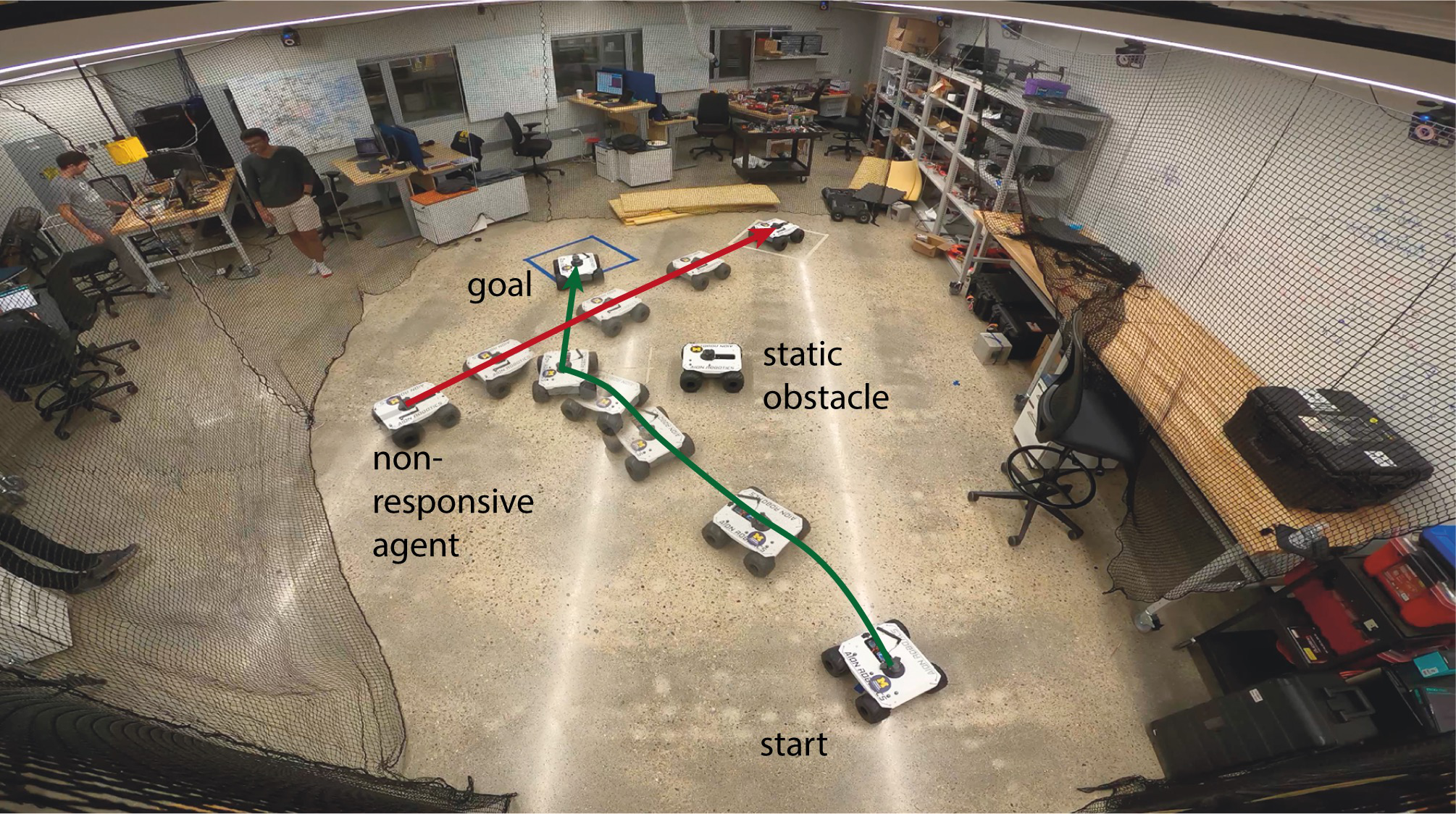}
    \caption{\small{A rover avoids a static and dynamic rover using our proposed C-CBF controller en route to a target in the laboratory setting.}}\label{fig.experiment}
\end{figure}

For the setup, the static rover was placed directly between the ego rover and its goal, while the dynamic rover was stationary until suddenly moving across the ego's path as it approached its target. As highlighted in Figure \ref{fig.experiment}, the ego rover first headed away from the static rover and then decelerated and swerved to avoid a collision with the second rover before correcting course and reaching its goal. Videos and code for both this experiment and the simulation in Section \ref{sec.numerical_case_study} is available on Github\footnote{Link to Github repo: \href{https://github.com/6lackmitchell/CCBF-Control}{github.com/6lackmitchell/CCBF-Control}}.


\section{Conclusion}\label{sec.conclusion}
In this paper, we addressed the problem of safe control under multiple state constraints via a C-CBF based control design. To ensure that the synthesized C-CBF is valid, we introduced a parameter adaptation law on the weights of the C-CBF constituent functions and proved that the resulting controller is safe. We then demonstrated the success of our approach on a multi-robot control problem in a crowded warehouse environment, and further validated our work on a ground rover experiment in the lab.

In the future, we plan to explore conditions under which the C-CBF approach may preserve guarantees in the presence of input constraints, including whether alternative adaptation laws for the weights assist in guarantees of liveness in addition to safety.

\bibliographystyle{IEEEtran}
\bibliography{root}


\appendices


\end{document}